\documentclass[12pt,reqno,a4paper]{amsart}
\usepackage{amsmath,amssymb,amsthm,amsaddr}
\usepackage{enumitem,float}
\usepackage[margin=2.5cm,top=3.0cm,bottom=3.0cm,footskip=1cm,headsep=1cm]{geometry}
\usepackage[utf8]{inputenc}
\usepackage{graphicx}
\usepackage{tikz}
\usepackage{mathtools}

\usetikzlibrary{arrows}

\author{H. Egger}
\address{Department of Mathematics, TU Darmstadt, Germany}
\email{egger@mathematik.tu-darmstadt.de}

\title[Variational discretization of non-isothermal compressible flow]{A mixed variational discretization for non-isothermal compressible flow in pipelines}

\newtheorem{lemma}{Lemma}[section]
\newtheorem{problem}[lemma]{Problem}

\newtheorem{assumption}[lemma]{Assumption}

\theoremstyle{definition}
\newtheorem{remark}[lemma]{Remark}
\newtheorem{notation}[lemma]{Notation}
\newtheorem{example}{Example}[section]
\newtheorem*{example*}{Example}

\def\div{\mathrm{div}}

\def\dt{\partial_t}
\def\dtau{\bar\partial_\tau}
\def\dx{\partial_x}

\def\dxx{\partial_{xx}}

\def\RR{\mathbb{R}}

\def\E{\mathcal{E}}

\numberwithin{equation}{section}
\numberwithin{table}{section}
\numberwithin{figure}{section}

\begin{document}

\begin{abstract} 
We consider the non-isothermal flow of a compressible fluid through pipes. 
Starting from the full set of Euler equations, we propose a variational characterization of solutions that encodes the conservation of mass, energy, and entropy in a very direct manner. This variational principle is suitable for a conforming Galerkin approximation in space which automatically inherits the basic physical conservation laws. Three different spaces are used for approximation of density, mass flux, and temperature, and we consider a mixed finite element method as one possible choice of suitable approximation spaces. We also investigate the subsequent discretization in time by a problem adapted implicit time stepping scheme for which exact conservation of mass as well as a slight dissipation of energy and increase of entropy are proven which are due to the numerical dissipation of the implicit time discretization. The main arguments of our analysis are rather general and allow us to extend the approach with minor modification to more general boundary conditions and flow models taking into account friction, viscosity, heat conduction, and heat exchange with the surrounding medium. 
\end{abstract}

\maketitle

\vspace*{-1em}

\begin{quote}
\end{quote}

\begin{quote}
\noindent
{\small {\bf AMS-classification (2000):}
35D30,35R02,37L65,76M10,76N99}
\end{quote}

\vspace*{1em}

\section{Introduction} \label{sec:intro}

We consider the systematic numerical approximation of compressible flow in pipelines. Such problems arise for instance in the simulation and optimization of gas networks \cite{BrouwerGasserHerty11,Osiadacz84}.
Motivated by such applications we tacitly assume throughout the text that the flow is subsonic.
For ease of presentation, we neglect for the moment the presence of friction, viscosity, and heat transfer, 
and therefore consider the Euler equations
\begin{align}
\dt \rho + \dx m &= 0,                     \label{eq:euler1}\\
\dt m + \dx \left(\frac{m^2}{\rho} + p\right) &= 0,   \label{eq:euler2}\\
\dt E + \dx \left(\frac{m}{\rho}(E+p)\right) &= 0.   \label{eq:euler3}
\end{align}
More general flow models will be considered below. 
Here $\rho$ denotes the density, $m$ the mass flux, $p$ the pressure, and $E = \frac{m^2}{2\rho} + \rho e$ the total energy; further $\frac{m^2}{2\rho}$ is the kinetic energy, and $e$ the internal energy.
The above equations are assumed to hold on a bounded closed interval $\omega$ representing the pipeline
and for all $t>0$. We further assume for the moment that the pipe is closed, i.e., 
\begin{align}
m&=0 \qquad \text{on the boundary}. \label{eq:euler4}
\end{align}
It is well-known that for smooth solutions of the Euler equations \eqref{eq:euler1}--\eqref{eq:euler3}, 
one can also deduce a conservation law for the entropy \cite{Feireisl03,NovotnyStraskraba04}, namely 
\begin{align} \label{eq:euler5}
\dt (\rho s) + \dx (m s) &= 0.
\end{align}
One may replace one of the equations in \eqref{eq:euler1}--\eqref{eq:euler3} by the entropy equation \eqref{eq:euler5} and thus obtain an equivalent system for the unknown density $\rho$, mass flux $m$, and temperature $\theta$. 
This viewpoint will play an important role in our considerations below.

The system \eqref{eq:euler1}--\eqref{eq:euler5} is complemented by \emph{equations of state}. 
Here we choose to characterize the pressure $p$, the specific internal energy $e$, and the specific entropy $s$ 
as functions of density $\rho$ and temperature $\theta$, i.e., 
\begin{align} \label{eq:euler6}
p=p(\rho,\theta), 
\qquad e=e(\rho,\theta),
\qquad \text{and} \qquad 
s=s(\rho,\theta).
\end{align}
In order to comply to the basic laws of thermodynamics, these functions have to satisfy certain compatibility conditions; see \cite{Batchelor,Feireisl03} and Section~\ref{sec:prelim} below.

\medskip 

Under appropriate assumptions on the initial data and the constitutive relations, the local existence of smooth solutions to the system \eqref{eq:euler1}--\eqref{eq:euler6} can be guaranteed \cite{Feireisl03,NovotnyStraskraba04}. 
In the context of gas pipelines, solutions are expected to remain smooth for all time, which can be explained by the stabilizing effect of friction \cite{MarcatiMilani90,MarcatiRubino00}. Moreover, the flow takes place at low Mach number and therefore no shocks should be generated. 
\begin{notation}
A triple $(\rho,m,\theta) \in C^1((0,T] \times \omega)^3$ is called \emph{smooth positive solution} of the system \eqref{eq:euler1}--\eqref{eq:euler6} above, if all equations hold in a point wise sense and $\rho,\theta > 0$. 
\end{notation}

Due to the many important applications, a vast amount of literature has been devoted to the study of numerical methods for compressible flow problems, and the one dimensional problem discussed above is typically used as a starting point.
Despite that fact, the convergence analysis for the flow of inviscid fluids is not completely settled, not even in one space dimension.
Finite volume methods are probably most widely used for the numerical approximation of compressible flow. 
While a rather complete convergence theory has been established for scalar conservation laws, only partial results are available concerning the analysis of finite volume methods for the Euler equations; we refer to \cite{Kroener,LeVeque02} and the references given there. 
Similar results hold for discontinuous Galerkin methods \cite{CockburnShu89c,CockburnShu89b} 
which can be understood as high-order generalization of finite volume schemes.
Using the stabilizing effect of viscosity, some truly implementable numerical schemes for the isothermal compressible Navier-Stokes equations have been shown to be globally convergent to weak solutions \cite{ZarnowskiHoff91,ZhaoHoff94,ZhaoHoff97}. These are however formulated in Lagrangian coordinates which is prohibitive for a possible extension to pipe networks.
In \cite{GallouetEtAl16,Karper13}, a globally convergent non-conforming finite element method for the isothermal compressible Navier-Stokes equations in Eulerian coordinates has been proposed and analyzed. This method is based on the approximation of the velocity and density field and makes use of several stabilization terms which 
leads to a rather strong violation of the conservation laws. Moreover, the method 
degenerates in the inviscid limit; see \cite[Sec.~3.1]{Karper13}.

The main goal of the current manuscript is to construct a numerical scheme that can handle general flow models for viscous and inviscid fluids. Moreover, the method should have good stability properties and preserve the basic conservation laws that encode the underlying the physical principles as good as possible. For this purpose, we consider here an extension of our previous work \cite{Egger16}, which dealt with the conforming Galerkin approximation of isentropic flow on networks. 
The extension to the full set of Euler equations and generalization thereof, which are the subject of this paper, will require some non-trivial extensions. 
In contrast to the isentropic setting considered in \cite{Egger16}, we are not yet able to systematically handle pipe networks
and such extension are therefore left as a topic for future research. 
One obstacle here is the formulation of appropriate coupling conditions at pipe junctions. In contrast to the isentropic case \cite{Egger16,MorinReigstad15,MorinReigstad15}, this issue seems not completely settled in the general case; we refer to \cite{BandaHertyKlar06,BressanEtAl15,ColomboMauri08,Garavello10,Herty08} for positive as well as negative examples.  

\medskip

The remainder of the manuscript is organized as follows:
In the first part of the paper, we discuss in detail the Euler equations on a closed pipe. 
We first recall some basic relations of equilibrium thermodynamics, and then introduce a variational characterization of smooth solutions which rather directly encodes the conservation of mass, energy, and entropy. 
In the second part, we investigate the numerical approximation of this variational principle in space by a conforming Galerkin method and then consider the subsequent discretization in time by a problem adapted implicit time stepping scheme. We prove strict conservation of mass, energy, and entropy for the semi-discretization and  
a slight dissipation of energy and monotonic increase in entropy for the fully discrete scheme. The proposed method therefore perfectly complies to the basic principles of thermodynamics.
In the third part of the paper, we investigate the extension of our approach to more general flow models.
The last part of the paper is devoted to numerical tests which illustrate the theoretical results and demonstrate the stability and conservation properties of the fully discrete schemes. Our research is motivated mainly by gas transport in pipelines, which takes place at low Mach number and therefore avoids the generation of shocks. We however also consider a shock tube problem to demonstrate the correct handling of shocks, rarefaction waves, and contact discontinuities that may arise in more general applications.

\section*{Part I: Analysis on the continuous level}

In the following three sections, we first review some basic relations of equilibrium thermodynamics and then present and analyze a particular variational principle for the Euler equations which will serve as the basis for our further considerations.

\section{Auxiliary results} \label{sec:prelim}

Let us first consider in a bit more detail the relations of pressure, internal energy, and entropy.
It is well known, see e.g. \cite{Batchelor,CourantFriedrichs48,Feireisl03}, that for thermodynamical consistency, the pressure $p$ has to be related to the specific internal energy $e$ by
\begin{align} \label{eq:erho}
e_\rho = \frac{1}{\rho^2} (p - \theta p_\theta).
\end{align}
Subscripts denote partial derivatives and functions may in general depend on $\rho$ and $\theta$. 
By integration with respect to $\rho$, we can then express the \emph{internal energy} as 
\begin{align} \label{eq:e}
e(\rho,\theta) = P(\rho,\theta) - \theta P_\theta(\rho,\theta) + Q(\theta),
\end{align}
Here $P$ denotes a \emph{pressure potential} defined by 
\begin{align} \label{eq:P}
P(\rho,\theta) = \int_1^\rho \frac{p(r,\theta)}{r^2} dr,
\end{align}
and accordingly we call $Q(\theta)$ the \emph{thermal potential} which is independent of $\rho$. 
The two potentials allow us to rewrite various other terms that arise in our analysis later on in a common form. 
The spatial derivative of the pressure, for instance, can be expressed as
\begin{align} \label{eq:dpdx}
\frac{1}{\rho} \dx p 
&= \frac{1}{\rho} \dx (\rho^2 P_\rho) 
 = \dx \big((\rho P)_\rho \big) - P_\theta \dx \theta.
\end{align}
Another important quantity is the \emph{specific enthalpy} $h$, which is given by 
\begin{align} \label{eq:h}
h=e + \frac{p}{\rho} 
&= P - \theta P_\theta + Q + \rho P_\rho 
 = (\rho P)_\rho - \theta P_\theta + Q. 
\end{align}
The temporal change of the internal energy can then be expanded as 
\begin{align} \label{eq:drhoedt}
\dt(\rho e) 
&= h \dt \rho + (e-h) \dt \rho + \rho \dt e 
 = h \dt \rho + \rho \dt e - \frac{p}{\rho} \dt \rho. 
\end{align}
The last two terms of this splitting turn out to be related to the \emph{specific entropy} $s$, 
which is again determined up to a constant by thermodynamic relations \cite{Batchelor,Feireisl03}, namely
\begin{align} \label{eq:srho}
\theta s_\rho = e_\rho - \frac{p}{\rho^2}
\qquad \text{and} \qquad
\theta s_\theta = e_\theta.  
\end{align}
A suitable entropy can then be found by integration and reads 
\begin{align} \label{eq:s}
s(\rho,\theta) = \int_1^\theta \frac{Q_\theta(t)}{t} dt - P_\theta(\rho,\theta). 
\end{align}
The time derivative of this entropy can again be expressed via the potentials as
\begin{align} \label{eq:dsdt}
\rho \theta \dt s
&= -\rho \theta P_{\theta\rho} \dt \rho + \rho (Q_\theta - \theta P_{\theta\theta}) \dt \theta \\
&=  \rho (P_\rho - \theta P_{\theta\rho})\dt \rho + \rho (P_\theta - P_\theta - \theta P_{\theta\theta} + Q_\theta) \dt \theta - \rho P_\rho \dt \rho 
 =\rho \dt e - \frac{p}{\rho} \dt \rho. \qquad  \notag
\end{align}
Note that the last term already appeared in formula \eqref{eq:drhoedt} for the time derivative of the internal energy.
The spatial derivative of the entropy can finally be expressed as
\begin{align} \label{eq:dsdx}
\theta \dx s 
&= \theta (s_\rho \dx \rho + s_\theta \dx \theta) 
 = \theta s_\rho \dx \rho + \theta s_\theta \dx \theta \\
&= -\theta P_{\theta\rho} \dx \rho + (P_\theta - (\theta P_\theta)_\theta + Q_\theta) \dx \theta 
 = \dx (Q - \theta P_\theta) + P_\theta \dx \theta. \notag
\end{align}
A combination of these two formulas and the Euler equations \eqref{eq:euler1}--\eqref{eq:euler3} shows, see Remark~\ref{rem:entropy} for details,
that the evolution of the entropy 
is governed by 
\begin{align} \label{eq:entropy}
\rho \dt s + m \dx s &= 0.
\end{align}
Together with \eqref{eq:euler1} this yields the conservation law \eqref{eq:euler5} for the entropy.
\begin{remark}
To completely describe the constitutive relations for the fluid under investigation, it is thus sufficient to prescribe the two potentials $P(\rho,\theta)$ and $Q(\theta)$. 
Pressure $p$, internal energy $e$, entropy $s$, and derivatives of these quantities can then be expressed in terms of these potentials. 
This will substantially simplify our analysis later on.
\end{remark}

Following physical intuition and to guarantee the well-posedness of the problem under investigation, we make the following structural assumptions about the potentials.

\begin{assumption}
Let $P\in C^3(\RR_+^2)$ and $Q \in C^2(\RR_+)$ and assume that $P_\rho \ge 0$, $(\rho P_\rho)_\rho \ge 0$, and $Q_\theta - \theta P_{\theta\theta} \ge \underline c_v $ for all $\rho,\theta>0$ and with some constant $\underline c_v>0$.
\end{assumption}
These assumptions will tacitly be utilized at several places in our analysis below and they are therefore assumed to hold for the rest of the manuscript.
To show that they are reasonable, let us briefly discuss a particular example that has been discussed in literature.
\begin{example}[Admissible state equations]
Let us define 
\begin{align*}
P(\rho,\theta) = \frac{c_\gamma}{\gamma-1} \rho^{\gamma-1}  + C(\rho) \theta + c' 
\qquad \text{and} \qquad Q(\theta) = \int_1^\theta c_v(t) dt + c''.
\end{align*}
with $c_\gamma \ge 0$, $\gamma > 1$, and such that $c(\rho)=\rho^2 C_\rho \ge 0$, 
$(\rho C_{\rho})_\rho \ge 0$, and $c_v(\theta) \ge \underline c_v > 0$. 
The constants $c',c''$ do not play a significant role in the following. 
Via \eqref{eq:e} and \eqref{eq:P}, we 
can then express the pressure and internal energy as
\begin{align*}
p(\rho,\theta) = c_\gamma \rho^\gamma + c(\rho) \theta
\qquad \text{and} \qquad 
e(\rho,\theta) = \frac{c_\gamma}{\gamma-1} \rho^\gamma + Q(\theta) + c'.
\end{align*}
This yields an extension of the state equations for an ideal gas that has already been considered in \cite{Feireisl03}. 
From the definition of the pressure potential, we can further see that
\begin{align*}
P_{\rho}(\rho,\theta) = c_\gamma \rho^{\gamma-2} + C_\rho(\rho) \theta  \ge 0  \qquad \text{for all } \rho,\theta>0
\end{align*}
and 
\begin{align*}
(\rho P_\rho)_\rho 
=  c_\gamma (\gamma-1) \rho^{\gamma-2} + (\rho C_\rho)_\rho \ge 0 \qquad \text{for all } \rho,\theta>0.
\end{align*}
In addition, we obtain
\begin{align*}
Q_\theta  - \theta P_{\theta\theta} = Q_\theta = c_v(\theta) \ge \underline c_v > 0 \qquad \text{for all } \theta>0.
\end{align*}
Hence all assumptions about the potentials made above are valid. Note that this simple example already covers, as special cases, the state equations of polytropic gases \cite{CourantFriedrichs48,LeVeque02} and of barotropic flow \cite{Feireisl03,NovotnyStraskraba04}. Our results are therefore directly applicable in these situations. 
\end{example}

\section{An equivalent formulation} \label{sec:equivalent}

We next present an equivalent formulation for the Euler equations which turns out to be particularly well suited for numerical approximation. Recall the definition of the total energy density $E=\frac{m^2}{2\rho} + \rho e$. 
With the formulas of the previous section, we then obtain
\begin{align} \label{eq:dEdt}
\dt E 
&= \dt \left(\frac{m^2}{2\rho^2}\right) + \dt (\rho e) \\
&= m \left(\frac{1}{\rho} \dt m - \frac{m}{2\rho^2} \dt \rho \right) + h (\dt \rho) + (\rho \dt e - \frac{p}{\rho} \dt \rho). \notag
\end{align}
Note that the last term could also be expressed in terms of the entropy by $\theta \rho \dt s$,
which will in fact be utilized below.
By the Euler equations \eqref{eq:euler1}--\eqref{eq:euler3} and the relations between the constitutive equations and the potentials derived in the previous section, the terms in parenthesis on the right hand side of \eqref{eq:dEdt} can be expressed as 
\begin{align}
\dt \rho &=- \dx m,\label{eq:sys1}\\
\frac{1}{\rho} \dt m - \frac{m}{2\rho^2} \dt \rho &=- \dx (\frac{m^2}{2\rho^2} + (\rho P)_\rho) - \frac{m}{2\rho^2} \dx m  + P_\theta \dx \theta, \label{eq:sys2}\\
\rho \dt e - \frac{p}{\rho} \dt \rho  &= - m \dx (Q - \theta P_\theta) - m P_\theta \dx \theta. \label{eq:sys3}
\end{align}
These equations provide an equivalent formulation for the problem under investigation.
\begin{lemma}[Equivalence] \label{lem:equivalence}
Let the state equations be defined as in Section~\ref{sec:prelim}. 
Then any smooth positive solution $(\rho,m,\theta)$ of the Euler equations \eqref{eq:euler1}--\eqref{eq:euler3} also solves \eqref{eq:sys1}--\eqref{eq:sys3}, and vice versa. The two systems are equivalent in this sense.
\end{lemma}
\begin{proof}
The first equation \eqref{eq:sys1} is the same as \eqref{eq:euler1}. \\
Next consider \eqref{eq:sys2}: From equations \eqref{eq:euler1}--\eqref{eq:euler2}, we can deduce that
\begin{align*}
\frac{1}{\rho} \dt m - \frac{m}{2\rho^2} \dt \rho
&= - \frac{1}{\rho} \dx \left(\frac{m^2}{\rho} + p\right) + \frac{m}{2\rho^2} \dx m = (i) + (ii) + (iii).
\end{align*}
The first term can be expanded as
\begin{align*}
-\frac{1}{\rho} \dx \left(\frac{m^2}{\rho}\right) 
&= -\frac{m}{\rho}\dx \left(\frac{m}{\rho}\right) - \frac{m}{\rho^2} \dx m
 = -\dx \left(\frac{m^2}{2\rho^2}\right) - \frac{m}{\rho^2} \dx m.
\end{align*}
The second term (ii) can be replaced by the expression for $\frac{1}{\rho} \dx p$ derived in \eqref{eq:dpdx}. 
Adding up all three terms then directly yields the second equation \eqref{eq:sys2}.\\
Now consider the third identity \eqref{eq:sys3}: From equation \eqref{eq:dEdt}, we immediately obtain
\begin{align*}
\rho \dt e - \frac{p}{\rho} \dt \rho 
&= \dt E - m \left(\frac{1}{\rho} \dt m - \frac{m}{2\rho^2} \dt \rho \right) - h \dt \rho = (*).
\end{align*}
We can now use \eqref{eq:euler3}, \eqref{eq:sys2}, and \eqref{eq:sys1} to replace the time derivative terms,
which yields
\begin{align*}
(*) &= - \dx \left( \frac{m}{\rho} (E+p)\right) + m \left( \dx \left( \frac{m^2}{2\rho^2} + (\rho P)_\rho \right) + \frac{m}{2\rho^2} \dx m - P_\theta \dx \theta \right) + h \dx m.
\end{align*}
Using that $\frac{E+p}{\rho}=\frac{m^2}{2\rho^2} + h$ and $h = (\rho P)_\rho - \theta P_\theta + Q$, we can expand the first term as 
\begin{align*}
- \dx \left( \frac{m}{\rho} (E+p)\right)
&= - m \dx \left(\frac{m^2}{2\rho^2}\right) -\dx m \left(\frac{m^2}{2\rho^2}\right)  - m \dx \big( (\rho P)_\rho - \theta P_\theta + Q \big) - h \dx m.
\end{align*}
A combination with the remaining terms then yields the third identity  \eqref{eq:sys3}. 
Reversing the individual steps yields the other direction and completes the proof of the lemma.
\end{proof}

\begin{remark} \label{rem:entropy}
Using the formulas \eqref{eq:dsdt} and \eqref{eq:dsdx} for the derivatives of the entropy, the third equation \eqref{eq:sys3} could also be expressed as 
\begin{align*}
\rho \theta \dt s = - m \theta \dx s.
\end{align*}
Since $\theta>0$ for any smooth positive solution, 
this is equivalent to \eqref{eq:entropy} and together with equation \eqref{eq:euler1} 
leads to the conservation law \eqref{eq:euler5} for the entropy.
The above three equations \eqref{eq:sys1}--\eqref{eq:sys3} thus actually describe the conservation of mass 
and the evolution of kinetic energy and of entropy. The unknown fields here are the density $\rho$, 
the mass flux $m$, and the temperature $\theta$. This is the basic framework for our further considerations.
\end{remark}

\section{A variational principle} \label{sec:var}

The equivalent formulation \eqref{eq:sys1}--\eqref{eq:sys3} allows us to establish the following variational characterization of solutions to the Euler equations which will be the starting point and main ingredient for the construction and analysis of numerical approximations later on.
\begin{lemma}[Variational characterization] $ $ \label{lem:variational}
Let the state equations be given as in Section~\ref{sec:prelim}. 
Then any smooth positive solution $(\rho,m,\theta)$ of \eqref{eq:euler1}--\eqref{eq:euler4}
also solves
\begin{align}
(\dt \rho,q) &+ (\dx m,q) = 0,  \label{eq:var1}\\
\left(\frac{1}{\rho} \dt \rho - \frac{m}{2\rho^2} \dt \rho,v\right) &- \left(\frac{m^2}{2\rho^2} + (\rho P)_\rho,\dx v\right) + \left(\frac{m}{2\rho^2} \dx m - P_\theta \dx \theta, v\right) = 0, \label{eq:var2}\\
\left(\rho \dt e - \frac{p}{\rho} \dt \rho,\frac{w}{\theta}\right) &- \left(Q - \theta P_\theta, \dx \big(m\frac{w}{\theta}\big)\right)  + \left(m P_\theta \dx \theta,\frac{w}{\theta}\right) = 0, \label{eq:var3}
\end{align}
for all test functions $q \in L^2(\omega)$, $v \in H_0(\div;\omega)$, $w \in H^1(\omega)$, and all $t \ge 0$.
Vice versa, any smooth positive solution of \eqref{eq:var1}--\eqref{eq:var3} and \eqref{eq:euler4} also solves \eqref{eq:euler1}--\eqref{eq:euler3}.
\end{lemma}
\begin{notation}
Here $(u,v) = \int_\omega u v dx$ denotes the standard scalar product of $L^2(\omega)$ and $H^1(\omega) = \{\theta \in L^2(\omega) : \dx \theta \in L^2(\omega)\}$ is the usual Sobolev space. In addition, we denote by $H_0(\div;\omega)=\{m \in L^2(\omega) : \dx m \in L^2(\omega) \text{ with } m=0 \text{ on } \partial\omega\}$ the space of smooth flux functions that additionally vanish at the boundary. 
Note that on networks \cite{Egger16} or in multiple dimensions, $H(\div)  \neq H^1$. 
We therefore use this specific notation already here. 
\end{notation}
\begin{proof}[Proof of Lemma~\ref{lem:variational}]
The assertion follows directly from the equivalence of the Euler equations
with \eqref{eq:sys1}--\eqref{eq:sys3}, by multiplying these latter set of equations with appropriate test functions, integration over the domain, and integration-by-parts for some of the terms. Note that all boundary terms vanish due to the homogeneous boundary conditions for $m$ and $v$. This shows that any solution of the Euler equations satisfies the above variational principle. The other direction follows by reversing the individual steps.
\end{proof}

From the previous two lemmas and the physical principles underlying the Euler equations, one can now immediately deduce global conservation laws for mass, energy, and entropy for all smooth positive solutions of the above variational principle.

\begin{lemma}[Global conservation] \label{lem:conservation} $ $\\
Let $(\rho,m,\theta)$ be a smooth positive solution of \eqref{eq:var1}--\eqref{eq:var3} and \eqref{eq:euler4}.
Then 
\begin{align*}
\frac{d}{dt} \int_\omega \rho dx = 0, 
\qquad  \frac{d}{dt} \int_\omega E dx =0,
\quad \text{and} \quad 
\frac{d}{dt} \int_\omega \rho s dx = 0.
\end{align*}
\end{lemma}
\begin{proof}
We provide a detailed proof that is only based on the particular form of the variational principle 
and which carries over directly to the discrete setting considered later on.
Conservation of mass follows by testing \eqref{eq:var1} with $q=1$ which yields
\begin{align*}
\int_\omega \dt \rho dx = (\dt \rho,1) = -(\dx m,1) = -\int_\omega \dx m dx = 0,
\end{align*}
where we used the boundary conditions \eqref{eq:euler4} in the last step.
The second identity is obtained as follows: 
Using formula \eqref{eq:dEdt} for the derivative of the total energy, we get
\begin{align*}
\frac{d}{dt} \int_\omega E dx 
&= (\dt E,1 ) \\
&= \left(\frac{1}{\rho} \dt m - \frac{m}{2\rho^2} \dt \rho,m\right) + \Big(\dt \rho, (\rho P)_\rho - \theta P_\theta + Q\Big) + \left(\rho \dt e - \frac{p}{\rho} \dt \rho,1/\theta\right).
\end{align*}
Testing the variational principle with $q=(\rho P)_\rho -\theta P_\theta + Q$, $v=m$, and $w=\theta$ leads to 
\begin{align*}
 \frac{d}{dt}\int_\omega E dx 
&= \left(\frac{m^2}{2\rho^2} + (\rho P)_\rho,\dx m\right) - \left(\frac{m}{2\rho^2} \dx m,m\right) + \big(P_\theta \dx \theta,m\big) \\
&\qquad \qquad 
- \big(\dx m, (\rho P)_\rho -\theta P_\theta + Q\big) 
+ \big(Q - \theta P_\theta,\dx m\big) - \big(P_\theta \dx \theta,m\big) 
= 0.
\end{align*}
Note that all terms vanish due to the particular structure of the variational principle which is thus responsible for the conservation of energy.
In order to verify the conservation of entropy, we proceed as follows: We start by observing that
\begin{align*}
\frac{d}{dt} \int_\omega \rho s dx 
&= (s \dt \rho + \rho \dt s,1) 
 = (\dt \rho, s) + \left(\theta \rho \dt s,\frac{1}{\theta}\right) 
 = (i) + (ii).
\end{align*}
Testing the first equation in the variational principle with $q=s$ leads to
\begin{align*}
(i) 
&= -(\dx m,s) = -\Big(\dx m, \int_1^\theta \frac{Q_\theta(t)}{t} dt - P_\theta\Big) 
 = \Big(m, \frac{1}{\theta} \dx Q\Big) + \big(\dx m, P_\theta\big).
\end{align*}
In the last step, we used integration-by-parts and the boundary conditions \eqref{eq:euler4} for $m$.
The expression \eqref{eq:dsdt} for the derivative of the entropy and the third equation in the variational principle tested with $w=1$ then allows us to express the second term as
\begin{align*}
(ii) 
&= \left(\rho \dt e - \frac{p}{\rho} \dt \rho,\frac{1}{\theta}\right)  
 = \left(Q,\dx \frac{m}{\theta}\right) - \left(\theta P_\theta, \frac{1}{\theta} \dx m - m \frac{1}{\theta^2} \dx \theta\right) - \left(m P_\theta \dx \theta, \frac{1}{\theta}\right) \\
&= - \left(\frac{1}{\theta} \dx Q, m \right) - \big(P_\theta, \dx m\big). 
\end{align*}
A summation of the two terms now directly yields the conservation of entropy. 
\end{proof}

\section*{Part II: Discretization}

In the following two sections, we discuss the discretization of the variational principle for the Euler equations. We start by a Galerkin approximation in space and then study the subsequent discretization in time by an implicit time stepping scheme.

\section{Galerkin approximation} \label{sec:galerkin}

For the numerical approximation in space, 
we consider a conforming Galerkin approximation of the variational principle \eqref{eq:var1}--\eqref{eq:var3} stated above. 
For this purpose, we choose finite dimensional subspaces $Q_h \subset L^2(\omega)$, $V_h \subset H_0(\div;\omega)$, and $W_h \subset H^1(\omega)$.
The semi-discrete solution is then defined by the following discrete variational problem. 
\begin{problem}[Galerkin semi-discretization] \label{prob:semi}  $ $\\
Let $\rho_{h,0} \in Q_h$, $m_{h,0} \in V_h$, $\theta_{h,0} \in W_h$ be given.
Find $(\rho_h,m_h,\theta_h) \in C([0,T];Q_h \times V_h \times W_h)$ such that $\rho_h(0)=\rho_{h,0}$, 
$m_h(0)=m_{h,0}$, $\theta_h(0)=\theta_{h,0}$, and such that
\begin{align*}
(\dt \rho_h,q_h) &+ (\dx m_h,q_h) = 0, \\
\left(\frac{1}{\rho_h} \dt m_h - \frac{m_h}{2\rho_h^2} \dt \rho_h,v_h\right) &- \left(\frac{m_h^2}{2\rho_h^2}+ (\rho_h P_h)_\rho,\dx v_h\right) + \left(\frac{m_h}{2\rho_h^2} \dx m_h - P_{\theta,h} \dx \theta_h, v_h\right) = 0, \\
\left(\rho_h \dt e_h - \frac{p_h}{\rho_h} \dt \rho_h,\frac{w_h}{\theta_h}\right) &- \left(Q_h - \theta_h P_{\theta,h}, \dx \Big(m_h \frac{w_h}{\theta_h}\Big)\right) + \left(m_h P_{\theta,h} \dx \theta_h,\frac{w_h}{\theta_h}\right) = 0,
\end{align*}
holds for all test functions $q_h \in Q_h$, $v_h \in V_h$, $w_h \in W_h$, and for all $t \ge 0$. 
\end{problem}
Here $P_h=P(\rho_h,\theta_h)$, $P_{\theta,h}=P_\theta(\rho_h,\theta_h)$, $P_{\rho,h}=P_\rho(\rho_h,\theta_h)$, and $Q_h=Q(\theta_h)$ denote the respective functions evaluated at the discrete solutions. Also note that the solution components $\rho_h$, $m_h$, and $\theta_h$ depend on $t$, while the test functions $q_h$, $v_h$, and $w_h$ are independent of time.
We next establish the local well-posedness of this problem.
\begin{lemma} \label{lem:wellposedh}
Let $\rho_{h,0} \ge \underline \rho > 0 $, $\theta_{h,0} \ge \underline \theta > 0$, and let $P,Q$ be given as in Section~\ref{sec:prelim}.
Then Problem~\ref{prob:semi} admits a unique local solution $(\rho_h,m_h,\theta_h)$ on $[0,T]$ for some $T>0$. 
\end{lemma}
\begin{proof}
By definition of the discrete internal energy, we have  
\begin{align*}
e_h=e(\rho_h,\theta_h)=P(\rho_h,\theta_h) - \theta_h P_\theta(\rho_h,\theta_h) + Q(\theta_h).
\end{align*}
Differentiation with respect to time thus yields 
\begin{align*}
\dt  e_h 
&= (P_{\rho,h} - \theta_h P_{\theta\rho,h}) \dt \rho + (P_{\theta,h} - P_{\theta,h} - \theta P_{\theta\theta,h} + Q_{\theta,h}) \dt \theta_h \\
&= (P_{\rho,h} - \theta_h P_{\theta\rho,h}) \dt \rho + (Q_{\theta,h} - \theta P_{\theta\theta,h}) \dt \theta_h.
\end{align*}
After choosing a basis for $Q_h$, $V_h$, and $W_h$, the semi-discrete problem thus leads to a system of ordinary differential equations of the form 
\begin{align} \label{eq:ode}
M(Y) Y' + F(Y) = 0,
\end{align}
where $Y=(\hat \rho,\hat m,\hat \theta)$ denotes the coordinate vector for the functions $\rho_h$, $m_h$, and $\theta_h$. Due to our assumptions on $P$ and $Q$, the function $F$ and the matrix $M$ are continuously differentiable as long as $\rho_h$ and $\theta_h$ are strictly positive. Moreover, $M(Y)$ has block triangular form and therefore is regular, if the diagonal blocks are regular. The diagonal blocks, on the other hand, correspond to the mass matrices for the function spaces $V_h$, $Q_h$, and $W_h$ with weight functions $1$, $\frac{1}{\rho_h}$, and $\frac{\rho_h}{\theta_h}(Q_\theta(\theta_h) - \theta_h P_{\theta\theta}(\rho_h,\theta_h))$. By our assumptions on the potentials, these functions are strictly positive as long as $\rho_h$ and $\theta_h$ are strictly positive. Therefore, the matrix $M(Y)$ is regular if $\rho_h$ and $\theta_h$ are strictly positive. Local existence of a unique solution then follows from the Picard-Lindelöf theorem.
\end{proof}
\begin{remark}
By inspection of the proof, one can see that the solution can be extended uniquely in time as long as it remains bounded and $\rho_h$ and $\theta_h$ stay strictly positive. 
\end{remark}

Due to the conforming Galerkin approximation and the particular form of the variational principle, conservation of mass, energy, and entropy also hold on the discrete level.
\begin{lemma} \label{lem:conservationh}
Let $(\rho_h,m_h,\theta_h)$ be a solution of Problem~\ref{prob:semi} with $\rho_h > 0$ and $\theta_h>0$. 
Then 
\begin{align*}
\frac{d}{dt}\int_\omega \rho_h dx=0, 
\qquad 
\frac{d}{dt} \int_\omega E_h dx = 0, 
\quad \text{and} \quad 
\frac{d}{dt} \int_\omega \rho_h s_h dx = 0. 
\end{align*}
Here $E_h = \frac{m_h^2}{2\rho_h} + \rho_h e(\rho_h,\theta_h)$ and $s_h = s(\rho_h,\theta_h)$ denote the discrete energy and entropy.
\end{lemma}
\begin{proof}
The assertions can be proved with literally the same arguments as already employed in the proof of Lemma~\ref{lem:conservation} on the continuous level. 
\end{proof}

\begin{remark}
The conservation of energy also allows to obtain appropriate bounds for the norms of $\rho_h$, $\theta_h$, and $m_h$. 
As a consequence, one can expect global existence of the solution, as long as $\rho_h$ and $\theta_h$ remain uniformly bounded away from zero. This may be used as a first step towards a complete convergence analysis of the proposed method.
\end{remark}

\section{Time discretization}  \label{sec:time} 

As a final step of our investigations for the flow on a single pipe, we now
discuss an appropriate discretization in time which preserves the underlying conservation laws as good as possible.
Given a time step $\tau>0$, we define $t^n = n \tau$ for $n \ge 0$, 
and we denote by 
\begin{align*}
\dtau d^n :=  \frac{d^n - d^{n-1}}{\tau} \quad \text{for } n \ge 0,
\end{align*}
the backward differences which are taken as approximations for the time derivatives. 
As will become clear below, adaptive time steps $\tau_n>0$ could be considered as well. 
\subsection{Fully discrete scheme}
For the time discretization of the Galerkin approximation stated in Problem~\ref{prob:semi}, 
we now consider the following implicit time stepping scheme.
\begin{problem}[Fully discrete method] \label{prob:full} 
Set $\rho_h^0=\rho_{0,h}$, $m_h^0=m_{h,0}$, and $\theta_h^0=\theta_{0,h}$ with initial values as in Problem~\ref{prob:semi}.
For $n \ge 1$ find $(\rho_h^n,m_h^n,\theta_h^n)\in Q_h\times V_h \times W_h$, such that 
\begin{align*}
\left(\dtau \rho_h^n,q_h\right) 
  &+ \left(\dx m_h^n,q_h\right)=0,  \\
\left(\frac{1}{\rho_h^{n-1}} \dtau m_h^n - \frac{m_h^n}{2(\rho_h^{n})^2} \dtau \rho_h^n,v_h\right) 
  &-  \left(\frac{(m_h^n)^2}{2(\rho_h^n)^2} + (\rho P)_{\rho,h}^n,\dx v_h\right) \\
  & +  \left(\frac{m_h^n}{2(\rho_h^n)^2} \dx' m_h^n - P_{\theta,h}^n \dx \theta_h^n, v_h\right) = 0, \\
\left(\rho_h^{n-1} \dtau e_h^n - \frac{p_h^n}{\rho_h^n} \dtau \rho_h^n,\frac{w_h}{\theta_h^n}\right) &- \left(Q_h^n - \theta_h^n P_h^n, \dx \Big(m_h^n \frac{w_h}{\theta_h^n}\Big)\right) + \left(m_h^n P_h^n \dx \theta_h^n,\frac{w_h}{\theta_h^n}\right) = 0,
\end{align*}
holds for all test functions $q_h \in Q_h$, $v_h \in V_h$, and $w_h \in W_h$.
\end{problem}
As before $P_h^n=P(\rho_h^n,\theta_h^n)$ and similar expressions denote the corresponding functions evaluated at the discrete solutions $\rho_h^n$ and $\theta_h^n$. 
The local well-posedness of this fully discrete scheme can be obtained with similar arguments as that of the semi-discrete problem.
\begin{lemma}
Let $(\rho_h^{n-1},m_h^{n-1},\theta_h^{n-1})$ be given with $\rho_h^{n-1} >0$ and $\theta_h^{n-1} >0$. Furthermore, let the state equations be defined as in Section~\ref{sec:prelim}.
Then for $\tau>0$ sufficiently small, the system in Problem~\ref{prob:full} has a locally unique solution $(\rho_h^n,m_h^n,\theta_h^n)$ with $\rho_h^n>0$ and $\theta_h^n>0$.
\end{lemma}
\begin{proof}
With similar arguments as in the proof of Lemma~\ref{lem:wellposedh}, one can see that the problem for the $n$th time step can be formulated in algebraic form as 
\begin{align*} 
M(Y^{n};Y^{n-1}) \dtau Y^n + F(Y^n) = 0. 
\end{align*}
For $\tau=0$ we may set $M(Y^{n};Y^{n-1}) = M(Y^{n-1})$ with matrix $M(Y)$ as in the proof of Lemma~\ref{lem:wellposedh}. There it was shown that $M(Y)$ is regular under the assumptions of the lemma. 
Moreover, $M(Y;\tilde Y)$ and $F(Y)$ are continuously differentiable with respect to their arguments.  
Existence of a locally unique solution then follows by the implicit function theorem.
For sufficiently small $\tau>0$, we can deduce the positivity of $\rho_h^n$ and $\theta_h^n$ from their continuous dependence on the time step size.
\end{proof}

We comment in more detail on the size of the admissible time step below. 
Before, let us summarize the basic stability and conservation properties 
of the fully discrete scheme. 
\begin{lemma}[Conservation of mass, dissipation of energy, and non-decrease of entropy] \label{lem:conservationhh} $ $\\
Let $(\rho_h^n,m_h^n,\theta_h^n)_{n \ge 0}$ denote a positive solution of Problem~\ref{prob:full}.
Then 
\begin{align*}
\int_\omega \rho_h^n dx = \int_\omega \rho_h^k dx, 
\quad 
\int_\omega E_h^n dx \le \int_\omega E_h^k dx,
\quad \text{and} \quad 
\int_\omega \rho_h^n s_h^n dx \ge \int_\omega \rho_h^k s_h^k dx 
\end{align*}
for all $0 \le k \le n$. As before $E_h^n$ and $s_h^n$ denote the functions evaluated at $\rho_h^n$, $m_h^n$, and $\theta_h^n$.
\end{lemma}
\noindent
We call the solution of the discrete problem \emph{positive}, if $\rho_h^k>0$ and $\theta_h^k>0$ for all $0 \le k \le n$.

\begin{proof}
It suffices to consider the case $k=n-1$; the results for $k < n-1$ follow by induction. 
Testing the first equation with $q_h=1$ yields the exact conservation of mass
\begin{align*}
\int_\omega \rho_h^n dx 
&= (\rho_h^n,1) = (\rho_h^{n-1},1) + \tau (\dtau \rho_h^n,1)
 = (\rho_h^{n-1},1) - \tau (\dx m_h^n,1) = \int_\omega \rho_h^{n-1} dx.
\end{align*}
In the last step we used here that $m_h^n=0$ on the boundary.\\
Next consider the energy balance: Using the definition of the total energy $E$, we obtain
\begin{align*}
\frac{1}{\tau} \big(E_h^n - E_h^{n-1}\big) 
&= \dtau E_h^n
 = \dtau \left(\frac{(m_h^n)^2}{2\rho_h^n}\right)
   + \left(e_h^n + \frac{p_h^n}{\rho_h^n}\right) \dtau \rho_h^n 
   + \left(\rho_h^{n-1} \dtau e_h^n - \frac{p_h^n}{\rho_h^n} \dtau \rho_h^n\right). 
\end{align*}
The first term in this expression can be further expanded as
\begin{align*}
\frac{1}{\tau} \left(\frac{(m_h^n)^2}{2\rho_h^n} - \frac{(m_h^n)^2}{2\rho_h^{n-1}}\right)
   + \frac{1}{\tau}\left(\frac{(m_h^n)^2}{2\rho_h^{n-1}} - \frac{(m_h^{n-1})^2}{2\rho_h^{n-1}}\right) 
\le -\frac{(m_h^n)^2}{2(\rho_h^n)^2} \dtau \rho_h^n  + \frac{m_h^n}{\rho_h^{n-1}} \dtau m_h^n.
\end{align*}
In the last step, we used that for any convex differentiable function $f(y)$ there holds 
\begin{align*}
f(y^n) - f(y^{n-1}) \le f'(y^n) (y^n - y^{n-1}).
\end{align*}
This was applied here to the convex functions $f(\rho)=\frac{(m_h^n)^2}{2\rho}$ and $f(m) = \frac{m^2}{2\rho_h^{n-1}}$, respectively.
Substituting $e_h^n + \frac{p_h^n}{\rho_h^n} = h_h^n$ in the second term, 
we obtain after integration over $\omega$ that
\begin{align*}
&\int_\omega E_h^n dx - \int_\omega E_h^{n-1} dx 
=  \left( E_h^n-E_h^{n-1},1 \right) \\
&\le \tau \left(\frac{1}{\rho_{h}^{n-1}} \dtau m_h^n - \frac{(m_h^n)^2}{2(\rho_h^n)^2}\dtau \rho_h^n,m_h^n\right) 
 + \tau \left( \dtau \rho_h^n, h_h^n\right) + \tau \left(\rho_h^{n-1} \dtau e_h^n - \frac{p_h^n}{\rho_h^n} \dtau \rho_h^n,1\right).
 \end{align*}
The time differences can be further evaluated by testing the fully discrete problem with test functions
$q_h=h_h^n$, $v_h = m_h^n$, and $w_h=\theta_h^n$, similar as in the proof of Lemma~\ref{lem:conservation}.
One can then see again that the right hand side above sums up to zero. 

The change of the discrete entropy can finally be expressed as 
\begin{align*}
\int_\omega \rho_h^n s_h^n dx - \int_\omega \rho_h^{n-1} s_h^{n-1} dx
&= (\rho_h^n s_h^n,1) - (\rho_h^{n-1} s_h^{n-1},1) \\
&= \left(\rho_h^n - \rho_h^{n-1}, s_h^n\right) + \left(\rho_h^{n-1}, s_h^n - s_h^{n-1}\right) 
 = (i) + (ii).
\end{align*}
By similar arguments as in the proof of Lemma~\ref{lem:conservation}, 
the term (i) can be further evaluated using the first equation in the fully discrete problem with $q_h=s_h^n$, leading to
\begin{align*}
(i) 
&= \big(\rho_h^n - \rho_h^{n-1}, s_h^n\big) 
 = \tau \big(\dtau \rho_h^n,s_h^n\big) 
 = \tau \Big(m_h^n, \frac{1}{\theta_h^n}\dx Q(\theta_h^n)\Big) + \tau \big(\dx m_h^n, P_{\theta,h}^n\big).
\end{align*}
To estimate the second term, we start by deriving a discrete analogue of \eqref{eq:dsdt}.
Using integration along a suitable path and the thermodynamic relations \eqref{eq:entropy}, we obtain 
\begin{align*}
s_h^n - s_h^{n-1} 
&= s(\rho_h^n,\theta_h^n) - s(\rho_h^{n-1},\theta_h^{n-1})\\
&= \int_{\rho_h^{n-1}}^{\rho_h^n} s_\rho(\rho,\theta_h^n) d\rho 
+ \int_{\theta_h^{n-1}}^{\theta_h^n} s_\theta(\rho_h^{n-1},\theta) d\theta  \\
&=  \int_{\rho_h^{n-1}}^{\rho_h^n} \frac{1}{\theta_h^n}  e_\rho(\rho,\theta_h^n) d\rho 
+ \int_{\theta_h^{n-1}}^{\theta_h^n} \frac{1}{\theta} e_\theta(\rho_h^{n-1},\theta) d\theta
- \int_{\rho_h^{n-1}}^{\rho_h^n} \frac{1}{\theta_h^n} \frac{p(\rho,\theta_h^n)}{\rho} \frac{1}{\rho} d\rho.
\end{align*}
We can now use that $1/\theta$ is decreasing in $\theta$, $1/\rho$ is decreasing in $\rho$, 
and $\frac{p}{\rho}=\rho P_\rho$ is increasing in $\rho$. 
Moreover, these functions and also $e_\theta$ are non-negative. Hence we get
\begin{align*}
s_h^n - s_h^{n-1} 
&\ge \int_{\rho_h^{n-1}}^{\rho_h^n} \frac{1}{\theta_h^n} e_\rho(\rho,\theta_h^n) d\rho 
- \int_{\rho_h^{n-1}}^{\rho_h^n} \frac{1}{\theta_h^n} \frac{p(\rho_h^n,\theta_h^n)}{\rho_h^n}\frac{1}{\rho_h^{n-1}} d\rho
+ \int_{\theta_h^{n-1}}^{\theta_h^n} \frac{1}{\theta_h^n} e_\theta(\rho_h^{n-1},\theta) d\theta \\
&= \frac{1}{\theta_h^n} \big(e_h^n - e_h^{n-1}\big) - \frac{1}{\rho_h^{n-1}} \frac{1}{\theta_h^n} \frac{p_h^n}{\rho_h^n} \big(\rho_h^n - \rho_h^{n-1}\big).
\end{align*}
By testing the fully discrete variational principle with $w_h=1$, we thus obtain
\begin{align*}
(ii) &= (\rho_h^{n-1}, s_h^n - s_h^{n-1} ) 
\ge \tau \Big(\rho_h^{n-1} \dtau e_h^n - \frac{p_h^n}{\rho_h^n} \dtau \rho_h^n, \frac{1}{\theta_h^n}\Big) \\
&=\tau\Big(Q_h^n - \theta_h^n P_{\theta,h}^n,\dx \big(\frac{m_h^n}{\theta_h^n}\big) \Big) - \tau \Big(m_h^n P_h^n \dx \theta_h^n,\frac{1}{\theta_h^n}\Big) 
= -\tau \Big(\frac{\dx Q_h^n}{\theta_h^n}, m_h^n\Big) - \tau \big(P_{\theta,h}^n, \dx m_h^n\big).  
\end{align*}
A combination of the two estimates for (i) and (ii) now yields the entropy inequality. 
\end{proof}

\begin{remark}
The previous lemma shows that the total mass of the discrete solution is conserved exactly for all time, while the total energy may be slightly decreasing and the discrete entropy may be slightly increasing due to numerical dissipation caused by the implicit time stepping scheme. The fully discrete scheme is therefore energy and entropy stable, and in perfect agreement with the second law of thermodynamics. 
\end{remark}
\begin{remark}
The energy identity can be used to obtain bounds for the discrete solution. 
As long as $\rho_h^{n-1}$ and $\theta_h^{n-1}$ are bounded away from zero sufficiently well, the time step in the fully discrete scheme can hence be chosen reasonable large. 
For appropriate initial values, we therefore expect well-posedness of the scheme for a uniform time step $\tau>0$ and for all $n \ge 0$, which is also observed in our numerical tests. 
Let us note that, as a consequence of the numerical dissipation of the implicit time stepping scheme, 
the fully discrete solution will actually converge to a steady state on the long run.
\end{remark}

\section*{Part III: Generalizations}

In the following two sections, we briefly discuss some possible extensions of our approach to more general flow models. 
It will turn out that such generalizations can be incorporated very naturally and analyzed without difficulties.

\section{Non-homogeneous boundary conditions} \label{sec:boundary}

The case of a closed pipe which was considered in the previous section 
was convenient for the analysis and allowed us to establish global conservation laws for mass, energy, and entropy. 
This situation is however of minor practical relevance. 
We therefore discuss now the possibility to incorporate more general boundary conditions.
Motivated by transport of gas in pipelines, we always assume that the flow is sub-sonic in the whole pipe. 
Let $v_1$, $v_2$ denote the start and end point of the pipe $\omega=[v_1,v_2]$. 
Then $v \in \{v_1,v_2\}$ is called \emph{inflow} vertex, if $n(v) m(v) < 0$, 
and \emph{outflow} vertex, if $n(v) m(v) \ge 0$. Here
\begin{align*}
n(v_1)=-1 \qquad \text{and} \qquad n(v_2)=1, 
\end{align*}
plays the role of a normal vector at the boundary of the domain
and $m$ is the mass flux.
From the general theory of hyperbolic equations one can then deduce that in the sub-sonic regime, two boundary conditions have to be defined at an inflow boundaries, while only one condition has to be prescribed at an outflow boundary.
The typical situations that may occur are depicted in Figure~\ref{fig:boundary}.
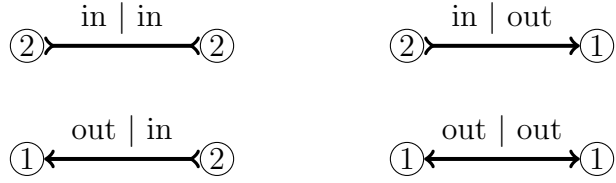
\begin{figure}[ht!]
\centering
\begin{tikzpicture}[scale=.5]
\node[circle,draw,inner sep=1pt] (v1a) at (0,3) {$2$};
\node[circle,draw,inner sep=1pt] (v2a) at (5,3) {$2$};
\node[circle,draw,inner sep=1pt] (v1b) at (10,3) {$2$};
\node[circle,draw,inner sep=1pt] (v2b) at (15,3) {$1$};
\node[circle,draw,inner sep=1pt] (v1c) at (0,0) {$1$};
\node[circle,draw,inner sep=1pt] (v2c) at (5,0) {$2$};
\node[circle,draw,inner sep=1pt] (v1d) at (10,0) {$1$};
\node[circle,draw,inner sep=1pt] (v2d) at (15,0) {$1$};
\draw[>-<,thick,line width=1.5pt] (v1a) -- node[above] {in $|$ in} ++ (v2a);
\draw[>->,thick,line width=1.5pt] (v1b) -- node[above] {in $|$ out} ++ (v2b);
\draw[<-<,thick,line width=1.5pt] (v1c) -- node[above] {out $|$ in} ++ (v2c);
\draw[<->,thick,line width=1.5pt] (v1d) -- node[above] {out $|$ out} ++ (v2d);
\end{tikzpicture}
\caption{Number of boundary conditions required for different flow situations. The closed pipe corresponds to two outflow vertices. \label{fig:boundary}}  
\end{figure}
The case of a closed pipe considered in the previous sections corresponds to that with two outflow boundaries here.
Let us emphasize that the choice of appropriate boundary conditions that give rise to a well-posed problem is not trivial;
an excellent review about appropriate choices can be found in \cite{Yee81}.

\medskip 

For later reference, let us briefly discuss one particular setting which will be utilized in our numerical tests below.
As before, let $\omega=[v_1,v_2]$ and  assume that 
\begin{align*}
m(v_1) = m_1^* > 0, \qquad \theta(v_1)=\theta_1^*, 
\qquad \text{and} \qquad m(v_2) = m_2^* > 0,
\end{align*}
with $m_1^*$ ,$m_2^*$, and $\theta_1^*$ given.
Here $v_1$ is the inflow boundary, where two conditions have to be prescribed, and $v_2$ is the outflow boundary, where only one condition is required. 
The following changes are needed in the variational principle \eqref{eq:var1}--\eqref{eq:var3}:

(i) We now have $m = m_0 + m^*$ with $m^0 \in H_0(\div)$ and $m^*$ given such that $m^*(v_i)=m^*_i$. 
The space for the corresponding test function $v$ does not have to be modified. 

(ii) To incorporate the additional boundary condition for the temperature, we can split 
$\theta=\theta_0 + \theta^*$ with $\theta_0 \in H_{0,1}^1 = \{w \in H^1(\omega) : w(v_1)=0\}$ and $\theta^*$ such that $\theta^*(v_1)=\theta^*_1$. The zero boundary condition at the inflow has also be required for the test function $w$. 

(iii) In the derivation of \eqref{eq:var3}, one obtains an additional boundary term 
\begin{align*}
(Q(\theta(v_2)) - \theta(v_2) P_\theta(\rho(v_2),\theta(v_2)) \frac{v(v_2)}{\theta(v_2)} m^*_2 
\end{align*}
which comes from integration-by-parts and where we replaced $m(v_2)  n(v_2)$ by $m^*_2$. This term does not vanish here, since neither $v$ nor $m$ vanish at the vertex $v_2$. 

Similar modifications are also required for the Galerkin approximation and fully discrete scheme. 
These can however easily be implemented on the algebraic level.

Due to the non-homogeneous boundary conditions, conservation of mass, energy, and entropy do no longer hold here in general. The change in total mass is given by
\begin{align*}
\frac{d}{dt} \int_\omega \rho \;dx 
= \int_\omega \dt \rho \; dx
= -\int_\omega \dx m \; dx 
= m_2^* - m_1^*.
\end{align*}
The total mass will thus still be conserved if the mass fluxes at the in- and outflow are of the same size. Note that this identity is again satisfied exactly by the Galerkin semi-discretization and also by the fully discrete scheme.

\section{More general flow models} \label{sec:general}

We next discuss the extension of our approach to more general flow models including the effects of viscosity, friction, heat conduction, and heat transfer across the pipe walls. 
The corresponding generalizations of the compressible Euler equations then read
\begin{align}
\dt \rho + \dx m &= 0, \label{eq:euler1a} \\
\dt m + \dx \left(\frac{m^2}{\rho} + p\right) &= a \rho \dx \left(\frac{1}{\rho^2} \dx m\right) - b \frac{|m|m}{\rho},   \label{eq:euler2a}\\
\dt E + \dx \left(\frac{m}{\rho}(E+p)\right) &= a m \dx \left(\frac{1}{\rho^2} \dx m\right) - \frac{b}{\rho^2} |m|^3 + \dx (c \dx \theta) + d (\theta^* - \theta).   \label{eq:euler3a}
\end{align}
Here $a,b,c,d \ge 0$ denote the coefficients for viscosity, friction, heat conduction, and heat transfer, and $\theta^*$ is the temperature of the surrounding medium.

\begin{remark}
The particular form of the viscous term has been proposed in \cite{Egger16} in the context of isentropic flow. 
For a constant density $\rho = \bar \rho$ and corresponding flux $m = \bar \rho u$, we have $\rho \dx (\frac{1}{\rho^2} \dx m) = \dxx u$ which is the usual viscosity term arising in the compressible Navier-Stokes equations \cite{Feireisl03,Lions98,NovotnyStraskraba04}.
Also note that the first two terms on the right hand side of the energy equation are sometimes neglected in the literature on gas networks; see for instance \cite{BrouwerGasserHerty11,Osiadacz84}. 
They are however required to obtain a consistent formulation.
\end{remark}
For convenience of notation, we again consider a pipe which is closed at the ends, i.e., 
\begin{align}
m &= 0 \quad \text{and} \quad c \dx \theta=0 \qquad \text{at the boundary}. \label{eq:euler4a}
\end{align}
The incorporation of more general boundary conditions will be discussed below.
As a direct consequence of the above system of equations, we again obtain a balance equation for the entropy, which now reads
\begin{align} \label{eq:euler5a}
\dt (\rho s) + \dx (m s) &=  \frac{1}{\theta} \dx (c \dx \theta) + \frac{d}{\theta} (\theta^* - \theta).
\end{align}
With similar reasoning as in Section~\ref{sec:var}, we deduce the following variational principle.
\begin{lemma}[Variational characterization] $ $ \label{lem:vara}
Let the state equations be given as in Section~\ref{sec:prelim}. 
Then any smooth positive solution $(\rho,m,\theta)$ of problem \eqref{eq:euler1a}--\eqref{eq:euler4a}
also solves 
\begin{align}
(\dt \rho,q) &+ (\dx m,q) = 0,  \label{eq:var1a}\\
\Big(\frac{1}{\rho} \dt m - \frac{m}{2\rho^2} \dt \rho,v\Big) &- \Big(\frac{m^2}{2\rho^2} + (\rho P)_\rho,\dx v\Big) + \Big(\frac{m}{2\rho^2} \dx m, v\Big) - \big(P_\theta \dx \theta, v\big) \label{eq:var2a} \\
& \qquad \qquad \ \ = - \Big(\frac{a}{\rho^2} \dx m, \dx v\Big) - \Big(\frac{b|m|}{\rho^2} m, v\Big),  \notag\\
\left(\rho \dt e - \frac{p}{\rho} \dt \rho,\frac{w}{\theta}\right) &- \left(Q - \theta P_\theta, \dx \Big(m \frac{w}{\theta}\Big)\right) + \left(m P_\theta \dx \theta,\frac{w}{\theta}\right)  
 \label{eq:var3a} \\& \qquad \qquad \ \ 
=  -\Big(c \dx \theta, \dx \big(\frac{w}{\theta}\big)\Big) + \Big(d(\theta^* - \theta) ,\frac{w}{\theta}\Big), \notag
\end{align}
for all test functions $q \in L^2(\omega)$, $v \in H_0(\div;\omega)$, $w \in H^1(\omega)$, and all $t \ge 0$.
Vice versa, any smooth positive solution of \eqref{eq:var1a}--\eqref{eq:var3a} and \eqref{eq:euler4a} also solves problem \eqref{eq:euler1a}--\eqref{eq:euler3a}.
\end{lemma}
The proof follows with minor modifications of that of Lemma~\ref{lem:variational} and is therefore omitted. 
Proceeding like in Section~\ref{sec:var}, we now obtain the following global balance laws.
\begin{lemma} \label{lem:conservationa}
Let $(\rho,m,\theta)$ be a smooth positive solution of \eqref{eq:var1a}--\eqref{eq:var3a} and \eqref{eq:euler4a}, and let the state equations be defined as in the Section~\ref{sec:prelim}. Then 
\begin{align*}
\frac{d}{dt} \int_\omega \rho dx = 0, \qquad 
&\frac{d}{dt} \int_\omega E dx = -\int_\omega \frac{a}{\rho^2} |\dx m|^2 + \frac{b}{\rho^2} |m|^3 + d (\theta - \theta^*) dx, \\
\quad \text{and}\qquad 
&\frac{d}{dt} \int_\omega \rho s dx = \int_\omega \frac{c}{\theta^2} |\dx \theta|^2 +  \frac{d}{\theta} (\theta^* - \theta) dx.
\end{align*}
\end{lemma}
\begin{proof}
The assertions follow with minor modifications of the proof of Lemma~\ref{lem:conservation}.
\end{proof}

\begin{remark}
Note that in comparison with the Euler equations, the additional effects of viscosity, friction, and heat conduction lead to dissipation of energy and to increase in entropy, respectively. Depending on the sign of $\theta - \theta^*$, the heat transfer through the pipe walls may yield a positive or a negative contribution to the global energy and entropy.
\end{remark}

We can now further proceed in the very same manner as outlined for the Euler equations in the previous sections. 
We skip the details but briefly comment on the main arguments. 
\begin{remark}
For the semi-discretization in space, we may use a Galerkin approximation as discussed in Section~\ref{sec:galerkin}. The global balance laws for mass, energy, and entropy stated in Lemma~\ref{lem:conservationa} are again inherited literally. 
For the discretization in time, we can also utilize the same strategy as outlined in Section~\ref{sec:time};
the additional terms appearing in the right hand side of \eqref{eq:var2a}--\eqref{eq:var3a} are treated implicitly. 
The fully discrete solution then again satisfies global balance laws for mass, energy, and entropy similar as the ones stated in Lemma~\ref{lem:conservationa}, but with the equalities replaced by respective inequalities in the energy and entropy balances; compare with Lemma~\ref{lem:conservationhh}.
These results follow with minor modifications of the proofs of Lemma~\ref{lem:conservationh} and \ref{lem:conservationhh} and details are therefore left to the reader. 
 \end{remark}

\section*{Part IV: Numerical validation}

In the final part of our paper, we now illustrate the theoretical results by some numerical tests which demonstrate the stability and performance of the proposed discretization. 

For all our computations, we utilize a Galerkin approximation with mixed finite elements
that we briefly introduce next.
Let $[0,\ell_e]$ be the interval related to the edge $e$ and
let $T_h(e) = \{K\}$ be a uniform partition of $e$ into elements $K$ of length $h$. 
The global mesh is defined as $T_h(\E) = \{T_h(e) : e \in \E\}$.
Next we define spaces of piecewise polynomials by
\begin{align*}
 P_k(T_h(\E)) &= \{v \in L^2(\E) : v|_e \in P_k(T_h(e)) \ \forall e \in \E\}. 
\end{align*}
Here $P_k(T_h(e)) = \{v \in L^2(e) : v|_K \in P_k(K), \ K \in T_h(e)\}$ denotes the space of piecewise polynomials over the mesh $T_h(e)$, and $P_k(K)$ 
is the space of polynomials of degree $\le k$ on the subinterval $K$. 
We then seek approximations for the density $\rho$, the mass flux $m$, and the temperature $\theta$, 
in the finite element spaces
\begin{align*} 
Q_h = P_{0}(T_h(\E)) \cap L^2, 
\quad 
V_h = P_{1}(T_h(\E)) \cap H_0(\div), 
\quad \text{and} \quad 
W_h = P_1(T_h(\E)) \cap H^1.
\end{align*}
These spaces are known to have good approximation properties. The first two spaces have already been used successfully for the numerical approximation of damped wave propagation problems and the isentropic Euler equations on networks in \cite{EggerKugler16,Egger16}.

\section{Numerical tests} \label{sec:num}

For illustration of our theoretical considerations, 
we now present some numerical results for two simple test problems.
In both cases, we consider the case of an ideal gas which in our language can be modeled by the potentials
\begin{align*}
P(\rho,\theta)=R \theta \log \rho \qquad \text{and} \qquad Q(\theta)=c_v \theta. 
\end{align*}
By simple computations and the formulas of Section~\ref{sec:prelim}, we then obtain 
\begin{align} \label{eq:state_equations}
p= R \theta \rho,
\qquad
e=c_v \theta,
\qquad 
h=c_p \theta, 
\qquad \text{and} \qquad
s=c_v \log \theta - R \log \rho.
\end{align}
Here $R$ is the gas constant, $c_v$ is the specific heat at constant volume, and $c_p=R+c_v$ is the specific heat at constant pressure. For the following tests, we set $R=1$, $c_v=2.5$ which corresponds to $c_p=3.5$ and an adiabatic coefficient of $\gamma=c_p/c_v=1.4$.

\subsection{A shock tube problem}

As a first test scenario, we consider the Sod problem \cite{Sod78} which
is known to exhibit a shock wave, a rarefaction wave, and a contact discontinuity. 
This allows us to demonstrate the stability and performance of our discretization scheme in a rather general situation. 
The flow is described by the Euler equations \eqref{eq:euler1}--\eqref{eq:euler6} on a pipe $\omega=[-2.5,2.5]$. 
As initial values, we choose 
\begin{align*}
\rho_0=\begin{cases} 1, & x<0, \\ 3, & x \ge 0,\end{cases}
\qquad
m_0=\begin{cases} 0, & x<0, \\ 0, & x \ge 0,\end{cases}
\qquad
\theta_0=\begin{cases} 1, & x<0, \\ 1, & x \ge 0,\end{cases}
\end{align*}
which corresponds to the setting considered in \cite[Sec.~14.13]{LeVeque02}.
In Figure~\ref{fig:sod}, we depict the numerical solution for this test problem at time $T=1$ obtained with the mixed finite element method with discretization parameters $h=\tau=1/100$. 
\begin{figure}[ht!]
\centering
\includegraphics[width=0.75\textwidth]{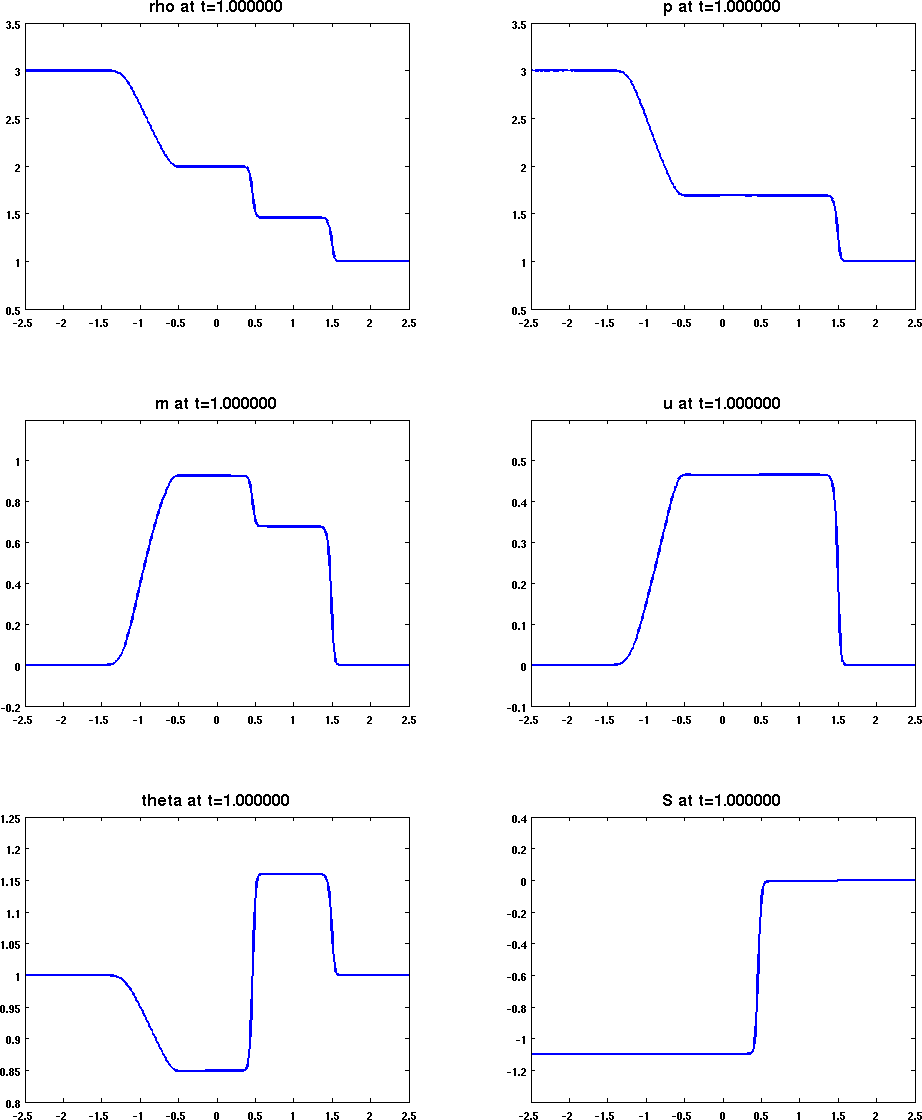}
\caption{Numerical solution for the Sod problem at time $T=1$ obtained with the mixed finite element method and discretization parameters $h=\tau=1/100$. The plots on the left correspond to the fields that are approximated directly in the numerical method. The plots on the right correspond to derived quantities. \label{fig:sod}} 
\end{figure}
In the left column, we depict the density $\rho$, the mass flux $m$, and the temperature $\theta$, which are the fields that are solved for in our numerical scheme. 
In all three solution components, one can see, from left to right, the rarefaction wave, the contact discontinuity, and the shock wave. The rarefaction wave leads to a smooth transition of the states, while the two latter features correspond to discontinuities which are slightly smeared out by the numerical scheme. The fronts however get sharper when further refining the mesh.  
The three pictures in the right column display the numerical approximations for the pressure $p$, the velocity $u$, and the specific entropy $s$, which are derived via \eqref{eq:state_equations} and the relation $m=\rho u$. 
As can be seen from the solution of the corresponding Riemann problem, the contact discontinuity does not appear in the pressure $p$ and velocity $u$; see \cite{Sod78} and \cite[Sec.13]{LeVeque02} for details.

As another validation of our theoretical results, let us also investigate the conservation of mass, energy, and entropy
in the numerical solution. For this, we repeat the previous test for different discretization parameters $h$ and $\tau$. 
The results are listed in Table~\ref{tab:sod}.
\begin{table}[ht!]
\begin{tabular}{c||r|r|r|r|r}
$h=\tau$                    & $1/20$     & $1/40$     & $1/80$     & $1/160$    & $1/320$ \\
\hline 
\hline 
$\triangle M_h$               & $ 0.0000$  & $ 0.0000$  & $ 0.0000$  & $ 0.0000$  & $ 0.0000$ \\
\hline
$\triangle E_h$               & $-0.0509$  & $-0.0400$  & $-0.0321$  & $-0.0268$  & $-0.0237$ \\
\hline
$\triangle S_h$               & $ 0.0797$  & $ 0.0549$  & $ 0.0384$  & $ 0.0276$  & $ 0.0207$
\end{tabular}
\medskip
\caption{Differences $\triangle M_h$, $\triangle E_h$, and $\triangle S_h$ in total mass, energy, and entropy 
of the numerical solution between time $t=1$ and $t=0$. \label{tab:sod}} 
\end{table}
As predicted by Lemma~\ref{lem:conservationhh}, the total mass is exactly conserved
on the discrete level.
Due to the numerical dissipation of the implicit time stepping scheme, the 
total energy is slightly decreasing and the total entropy is slightly increasing. 
The deviation from the exact conservation of energy and entropy, which is valid on the continuous and semi-discrete level, 
can be made smaller by decreasing the mesh size $h$ and the time step $\tau$. 
This can in fact already be observed when only the time step size $\tau$ is decreased.

\subsection{Gas transport through a pipe}

As a second test problem, we consider the transport of gas through a long pipeline. 
Here friction plays a major role for the dynamics and we also take into account 
heat exchange across the pipe walls.

The pipe is again modeled by the interval $\omega=[-2.5,2.5]$ and the evolution is now governed by 
the generalized flow model \eqref{eq:euler1a}--\eqref{eq:euler4a} with model parameters 
\begin{align*}
a=0, \qquad b=20, \qquad c=0, \qquad d=5, \quad \text{and} \quad  \theta^*=1.
\end{align*}
The state equations are chosen as before with parameters $\gamma=1.4$ and $R=1$. 
We assume that the fluid is at rest before $t=0$ and choose as initial conditions 
\begin{align*}
\rho_0=3, \qquad m_0=0, \qquad \text{and} \qquad \theta_0=1. 
\end{align*}
For time $t>0$, gas is injected at the left end and the same amount is drained at the right end of the pipe.
This is modeled by the boundary conditions
\begin{align*}
m=0.3, \quad  \theta=1.2 \ \text{ at } x=-2.5 \qquad \text{and} \qquad m=0.3 \ \text{ at } x=2.5.
\end{align*}
This setting corresponds to the one discussed in Section~\ref{sec:boundary}. 
From a simple dimension analysis one can deduce that the resulting flow is friction dominated and almost isothermal,
which is the typical setting observed in gas pipelines \cite{BrouwerGasserHerty11,Osiadacz84}. 
In Figure~\ref{fig:pipeline}, we display a few snapshots of the numerical solution
obtained with the fully discrete scheme based on the mixed finite element method with 
mesh size $h=1/100$ and time step $\tau=1/100$. 

\begin{figure}[ht!]
\centering
\includegraphics[width=0.65\textwidth]{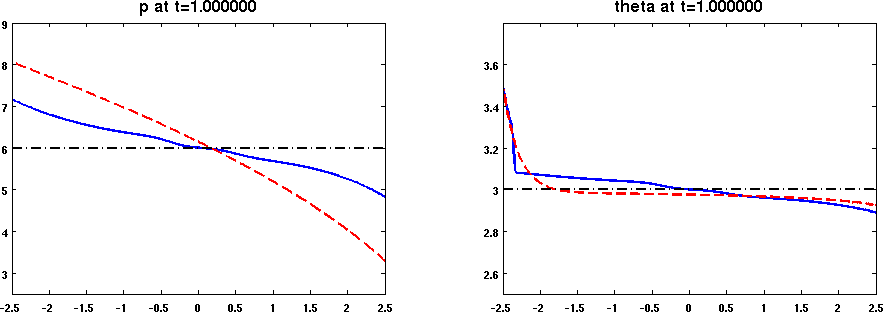} \\[1em]
\includegraphics[width=0.65\textwidth]{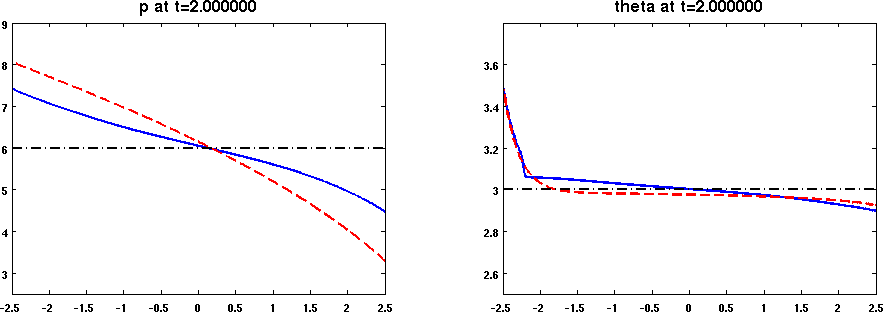} \\[1em]
\includegraphics[width=0.65\textwidth]{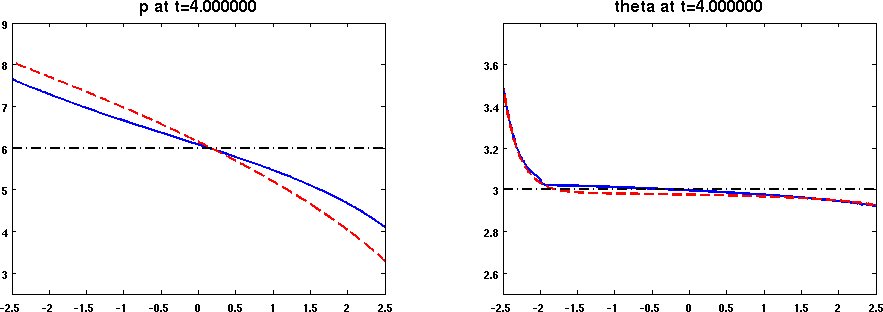} \\[1em]
\includegraphics[width=0.65\textwidth]{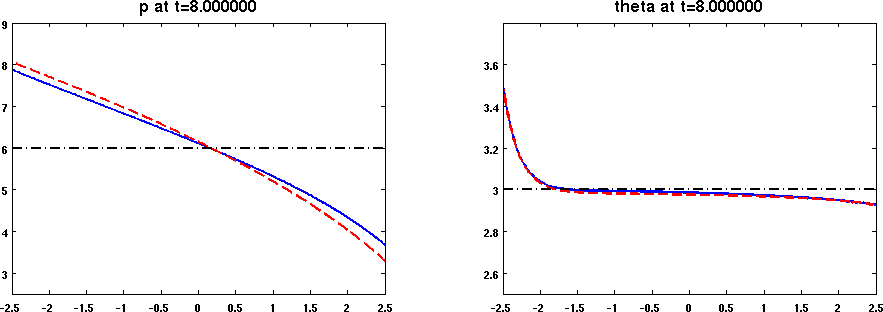} \\[1em]
\includegraphics[width=0.65\textwidth]{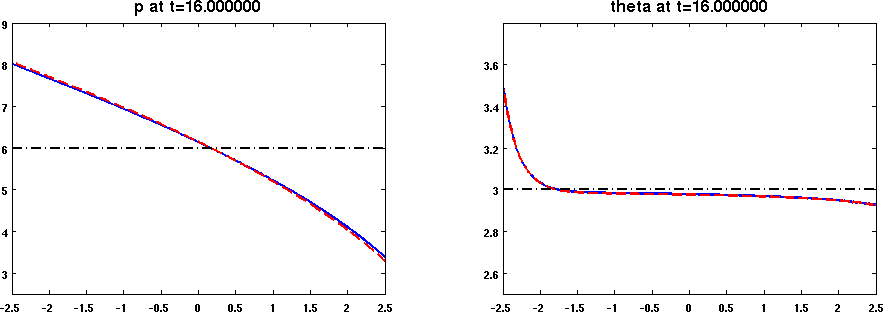} 
\caption{Initial values (black, dash-dot), steady states (red, dashed), and solution for the gas transport in a long pipeline at time $t=1,2,4,8,16$ (blue) obtained with the mixed finite element method and discretization parameters $h=\tau=1/100$. The plots on the left denote the pressure field and those on the right the temperature. \label{fig:pipeline}} 
\end{figure}

As a consequence of the damping and since the boundary conditions do not change with time, 
the system here converges to a new steady state $(\rho_\infty,m_\infty,\theta_\infty)$, 
governed by the corresponding stationary problem, which here reads
\begin{align*}
\dx m_\infty &= 0, \\
\dx \left(\frac{m_\infty^2}{\rho_\infty} + p(\rho_\infty,\theta_\infty) \right) &= -b \frac{|m_\infty|}{\rho_\infty} m_\infty, \\
\dx \left(m_\infty \left( \frac{m_\infty^2}{2\rho_\infty^2} + h(\rho_\infty,\theta_\infty)\right)\right) &= -b \frac{|m_\infty|^3}{\rho_\infty^2} + d (\theta^* - \theta_\infty). 
\end{align*}
In addition, we have 
\begin{align*}
m_\infty=0.3, \quad  \theta_\infty=1.2 \text{ at } x=-2.5 \qquad \text{and} \qquad m_\infty=0.3 \text{ at } x=2.5.
\end{align*}
Since $\dx m_\infty=0$, one of the boundary conditions is redundant here 
and therefore the stationary system does not completely determine the steady state. 
Due to conservation of mass, we however additionally get as another condition
\begin{align*}
\int_\omega \rho_\infty dx = \int_\omega \rho_0 dx 
\end{align*}
here which together with the previous equations uniquely determines the steady state. 

Due to exact conservation of mass on the discrete level, the discrete solution converges 
with $t \to \infty$ to an approximation of the correct steady which illustrates that the
exact conservation of mass is very important for the long-term behavior.
In Table~\ref{tab:pipeline}, we display the distance to steady state for the three solution components. 

\begin{table}[ht!]
\begin{tabular}{c||c|c|c|c|c|c}
$t$                 & $1$       & $2$       & $4$       & $8$       & $16$      & $32$ \\
\hline 
\hline 
$\triangle \rho$    & $0.6190$  & $0.4730$  & $0.3117$  & $0.1426$  & $0.0318$  & $0.0017$\\
\hline
$\triangle m$       & $0.3560$  & $0.1629$  & $0.0986$  & $0.0424$  & $0.0091$  & $0.0005$ \\
\hline
$\triangle \theta$  & $0.0916$  & $0.0719$  & $0.0422$  & $0.0183$  & $0.0041$  & $0.0002$
\end{tabular}
\medskip
\caption{Distances $\triangle \rho:=\|\rho_h(t)-\rho_{h,\infty}\|$, $\triangle m:=\|\rho_h(t)-\rho_{h,\infty}\|$, and 
$\triangle \theta:=\|\rho_h(t)-\rho_{h,\infty}\|$ to steady state for the three solution components. \label{tab:pipeline}} 
\end{table}
From the numerical results, one can deduce an exponential convergence to equilibrium. 
This could in principle be proven rigorously here by a linearized stability analysis; 
we refer to \cite{CoxZuazua94,EggerKugler15} for related results in a simplified setting.

\section{Discussion}

In this paper we proposed and analyzed the systematic discretization of compressible flow problems on a pipe
by Galerkin approximation in space and a problem adapted implicit time integration scheme. 
Exact conservation of mass, energy, and entropy could be proven for the semi-discretization of the Euler equations.
For the fully discrete scheme, exact conservation of mass is still preserved, while a slight decay in energy and a slight 
increase in entropy may be observed due to numerical dissipation of the implicit time stepping scheme. 
These properties and the stability of the scheme in the presence of shocks, rarefaction waves, and contact discontinuities was demonstrated by numerical results for a shock tube problem. For a problem involving high friction and heat exchange, which is the typical setting observed in gas transport through pipelines, we demonstrated the convergence to the correct quasi-steady state. 

The numerical method discussed in this paper is a generalization of the one for isentropic flow presented in \cite{Egger16}. 
There the method could be extended to problems on pipe networks, which is of relevance in the simulation of gas networks.
We were not successful yet to provide a corresponding extension to networks for problems and methods discussed in this paper.
One obstacle here is the formulation of appropriate coupling conditions at pipe junctions. Such extension are therefore left for future research.

\section*{Acknowledgements}
This work was supported by the German Research Foundation (DFG) via grants IRTG~1529 and TRR~154, and by the ``Excellence Initiative'' of the German Federal and State Governments via the Graduate School of Computational Engineering GSC~233.

\end{document}